\documentclass[a4paper,twosides,10pt,reqno]{article}

\usepackage{amsmath,amsthm,amsfonts,amssymb,amsopn,color}

\usepackage{amssymb,amsfonts}
\usepackage[all,arc]{xy}
\usepackage{esint,enumerate}
\usepackage{mathrsfs}
\usepackage{amsthm}
\usepackage{lipsum}
\usepackage{pgf,tikz}
\usetikzlibrary{arrows}
\usepackage{graphicx,mfpic}
\usepackage{subfigure}
\usepackage[utf8]{inputenc}
\usepackage{hyperref}

\makeatletter
\g@addto@macro\normalsize{%
\setlength\abovedisplayskip{4pt}
\setlength\belowdisplayskip{4pt}
\setlength\abovedisplayshortskip{4pt}
\setlength\belowdisplayshortskip{4pt}
}
\makeatother

\newtheorem{theorem}{Theorem}[section]
\newtheorem{lemma}[theorem]{Lemma}

\newtheorem{proposition}[theorem]{Proposition}

\theoremstyle{definition}

\newtheorem{remark}[theorem]{Remark}

\newcommand{\rone}{\mathbb{R}}
\newcommand{\rtwo}{{\mathbb R^2}}

\newcommand{\cpx}{\mathbb C}
\newcommand{\nat}{\mathbb N}

\newcommand{\calA}{{\mathcal A}}

\newcommand{\calF}{{\mathcal F}}

\newcommand{\calL}{{\mathcal L}}

\newcommand{\calX}{{\mathcal X}}

\newcommand{\tilF}{{\widetilde F}}

\newcommand{\tilu}{{\widetilde u}}

\newcommand{\tilphi}{{\widetilde \phi}}

\newcommand{\tilvp}{{\widetilde \varphi}}

\newcommand{\dist}{{\mathrm{dist }}}

\newcommand{\supp}{{\mathrm{supp }}}

\newcommand{\Uone}{{\mathrm{U}(1)}}
\newcommand{\SUn}{{\mathrm{SU}(n)}}

\newcommand{\al}{{\alpha}}
\newcommand{\e}{{\varepsilon}}
\newcommand{\ga}{{\gamma}}

\newcommand{\la}{{\lambda}}
\newcommand{\ve}{{\varepsilon}}
\newcommand{\vp}{{\varphi}}

\newcommand{\Om}{{\Omega}}

\newcommand{\pa}{\partial}

\newcommand{\qand}{{\quad \mbox{and} \quad}}
\newcommand{\qas}{{\quad \mbox{as} \quad}}
\newcommand{\qfor}{{\quad \mbox{for} \quad}}

\newcommand{\qin}{{\quad \mbox{in} \quad}}
\newcommand{\qon}{{\quad \mbox{on} \quad}}

\newcommand{\qwhere}{{\quad \mbox{where} \quad}}

\begin{document}

\title{\textbf{On a system of  multi-component Ginzburg-Landau vortices}}

\author{Rejeb Hadiji  \thanks{Univ Paris Est Creteil, CNRS, LAMA, F-94010 Creteil, France. \newline \indent \,~
 Univ Gustave Eiffel, LAMA, F-77447 Marne-la-Vall\'ee, France. \newline \indent \,~
 email: rejeb.hadiji@u-pec.fr.} $\,$ and $\,$
 Jongmin Han\thanks{Department of Mathematics, Kyung Hee University, 
 Seoul, 02447, Korea.\newline \indent \,~
 email:  jmhan@khu.ac.kr.} $\,$ and $\,$
Juhee Sohn\thanks{College of General Education,  Kookmin University,  Seoul  02707, Korea. \newline \indent \,~
 email: jhson37@kookmin.ac.kr.}
}
 \date{}
\maketitle

\begin{abstract}
We study the asymptotic behavior of solutions for $n$-component Ginzburg-Landau equations as $\ve \to 0$.
We prove that the minimizers converges locally in any $C^k$-norm to a solution of a system of   generalized harmonic map equations.
\end{abstract}

\small
\noindent{MSC2000 : 35B40, 35J60, 35Q60.}

\noindent{Keywords : $n$-component Ginzburg-Landau equations, semi-local gauge field model, asymptotic behavior of solutions }

\normalsize

\section{Introduction}
\setcounter{equation}{0}

The classical Ginzburg-Landau model   describes  the macroscopic theory for phenomena on superconductivity at low temperature.
It is derived from the Helmholtz free energy that consists of a complex order parameter $ u $ and the magnetic potential $A$ \cite{DGP92}.
In the physical literature, $| u |^2$ represents the density of the superconducting electron
pairs  in the superconducting material.
The state $| u |^2=0$ implies that  material remains in normal conducting state.
Whereas   the material  is superconducting if $| u |^2=1$.

The Ginzburg-Landau theory also gives a good insight for understanding various topological defects arising from cosmology.
The superconducting state is achieved at the vacuum level of the potential and it can be regarded as a broken symmetry during a phase transition.
For instance, this idea is realized by the (special) relativistic extension of the planar  Ginzburg-Landau theory, called the  Abelian-Higgs model \cite{NO73}.
 This model describes a charged scalar field $ \psi $ interacting with a $\Uone$ gauge  field $A$, and allows vortex solutions which are charged magnetically  but  electrically neutral.

The   Abelian-Higgs  model  is considered in the (2+1)-dimensional Minkowski
space $\rone^{1,2}$  with the metric  $\text{diag}(-1,1,1)$. The metric is used to raise or lower indices.
The Lagrangian density for the model is defined as
\begin{equation}
\label{eq:L0}
\calL_0=-\frac{1}{4}F_{\mu\nu}F^{\mu\nu}+\frac{1}{2}D_\mu u \overline{D^\mu u }-\frac{1}{8\ve^2}(| u |^2-1)^2.
\end{equation}
Here, $\mu,\nu=0,1,2 $, $\ve>0$ is  the Higgs coupling constant, $ u :\rone^{1,2}\to\cpx$ is the Higgs field and $A_\mu:\rone^{1,2}\to\rone$ is $\Uone$ gauge fields.
In addition,  $F_{\mu\nu}=\partial_\mu A_\nu-\partial_\nu A_\mu$ represents the electromagnetic field, $D_\mu u =\partial_\mu u +i A_\mu u $ is the covariant derivative.
The Lagrangian $\calL_0$ is invariant under the local $\Uone$ gauge transform
\[u\mapsto e^{i\alpha} u , \quad A_\mu\mapsto A_\mu- \partial_\mu\alpha
\]
for any smooth function $\alpha:\rone^{1,2}\to\rone$.
The Euler-Lagrange equations are
\begin{equation}
\label{eq:EL}
\left\{ \begin{aligned}
D_\mu D^\mu u &~=\frac{1}{ \ve^2}(| u |^2-1)u\\
\partial^\nu F_{\mu\nu}&~=\frac{i}{2\ve^2}\big[ u \overline{D_\mu u }-\overline{ u }D_\mu u \big],
\end{aligned}
\right.
\end{equation}
which are the (special) relativistic Ginzburg-Landau equations.

As a direct generalization of \eqref{eq:L0}, Vachaspati and  Achucarro proposed a new model in \cite{VaAch91}   that contains $n$ scalar fields:
\[\Psi= \big( u_1, \cdots u_n \big), \qwhere   u_k:\rone^{1,2}\to\cpx \qfor k=1,\cdots,n.
\]
The Lagrangian $\calL_0$ is modified as
\begin{equation}
\label{eq:L}
\calL=-\frac{1}{4}F_{\mu\nu}F^{\mu\nu}+\frac{1}{2}D_\mu\Psi\overline{D^\mu\Psi}-\frac{1}{8\ve^2}\big(|\Psi|^2-n\big)^2.
\end{equation}
It is easy to show that $\calL$ is invariant under the local $\Uone$ gauge transformation
\[\Psi\mapsto e^{i\alpha}\Psi,\quad A_\mu \mapsto A_\mu- \pa_\mu \alpha
\]
for any smooth function $\alpha:\rone^{1,2}\to\rone$.
It is also invariant under the global $\SUn$ gauge transformation
\[\Psi\mapsto e^{i\alpha_k\tau^k}\Psi,
\]
where $\alpha_k\in\rone$ for $k=1,\cdots,n^2-1$ and  $\{\tau^k\}_{k=1}^{n^2-1}$ are the generators of the Lie algebra $\mathfrak{su}(n)$.
The Euler-Lagrange equations are obtained as
\begin{equation}
\label{eq:GL-n}
\left\{ \begin{aligned}
D_\mu D^\mu u_k&~=\frac{1}{  \ve^2}  \Big( \sum_{j=1}^n |u_j|^2- n \Big) u_k,\\
\partial^\nu F_{\mu\nu}&~=\frac{i}{2\ve^2} \sum_{j=1}^n  \big(u_j\overline{D_\mu u_j}-\overline{u}_jD_\mu u_j\big),
\end{aligned}
\right.
\end{equation}
which we will refer to as $n$-component  Ginzburg-Landau equations.

The model  \eqref{eq:L} is useful in the study of some issues in cosmology, for instance, the formation of cosmic strings that have both local and global natures \cite{VaAch91}.
So, the string solutions for \eqref{eq:L} are called semilocal.
The semilocal string solutions for \eqref{eq:L} are similar to those for \eqref{eq:L0} that reflect the local gauge transformation.
But  they have additional features that have some resemblance to global defect.
For further physical implication of the model \eqref{eq:L}, one may refer to  \cite{Hi92,VaAch91}.

In this article, we are interested in static solutions for \eqref{eq:EL} and \eqref{eq:GL-n}.
In particular, we assume that the electromagnetic fields vanish, that is, $A_\mu=0$ for $\mu=0,1,2$.
Then, the main topic is to study the asymptotic behavior of solutions for \eqref{eq:EL} and \eqref{eq:GL-n} as $\ve \to 0$.
For the case \eqref{eq:EL},  there have been lots of researches on this topic after the seminal work of Bethuel-Brezis-Helein \cite{BBH93,BBH94}.
More specifically, let $\Om \subset \rtwo$ be a smooth bounded simply connected domain.
Given  a smooth map
\[ g  : \pa \Om \to S^1=\{ z \in \cpx : |z|=1\} \quad\text{with}\quad d=\deg (g,\pa\Om) 
\]
the   Ginzburg-Landau equations  \eqref{eq:EL} with $A_\mu=0$ reduce to
\begin{equation}
\label{eq:calssical GL}
\left\{
\begin{aligned}
-\Delta u& = \frac{1}{\ve^2} u (1-|u|^2 ) \qin \Omega, \\
u&=g  \qon \pa\Omega.
\end{aligned}
\right.
\end{equation}
The associated energy  functional is
\begin{equation}
\label{eq:ftnal Eb}
E_{\ve,\Om}^b (u) = \frac{1}{2} \int_\Om |\nabla u|^2  dx + \frac{1}{4\ve^2} \int_\Om (1-|u|^2)^2 dx .
\end{equation}
Then, $E_{\ve,\Om}^b$ has a minimizer $u_{\ve,g}^b$ in
\[ H^1_g (\Om; \cpx)= \big\{ u \in   H^1(\Om; \cpx) : u=g  \mbox{ on } \pa \Om \big\},
\]
that is
\begin{equation}\label{eq:BBH minimizer}
E_{\ve,\Om}^b (u_{\ve,g}^b) = \inf \{ E_{\ve,\Om}^b (u) : u \in H^1_g(\Om; \cpx)\}.
\end{equation}
We often write $H^1_g $ instead of $H^1_g (\Om; \cpx)$ if there is no confusion.

The asymptotic behavior of   solutions of \eqref{eq:calssical GL} as $\ve \to 0$
  has attracted lots of interest for three decades.
  The behavior of global minimizers for $u_{\ve,g}^b$ was studied in detail by
Bethuel, Brezis and H\'{e}lein \cite{BBH93,BBH94}.
If $d=0$, $u_{\ve,g}^b$ converges to a harmonic map that minimizes
\[  \int_\Om | \nabla u |^2dx
\]
over the space
\[  H^1_g (\Om;S^1) = \{ u \in W^{1,2} (\Om , S^1 )  :  u=g  \mbox{ on } \partial \Om  \}  .
\]
This problem has a solution $u_0^b$ satisfying
\begin{equation}
\label{eq:u0 GL}\left\{
\begin{aligned}
-\Delta u_0^b& = u_0^b |\nabla u_0^b|^2 \quad \text{on}\quad \Om,\\
u_0^b&=g  \qon \pa\Omega, \\
|u_0^b|&=1 \quad\text{ on } \Om.
\end{aligned}
\right.
\end{equation}
When $d \neq 0$, the analysis is more delicate because $ H^1_g (\Om;S^1) = \emptyset $.
There exists a set $\{ a_1, \cdots , a_d \} \subset \Om $ such that
up to a subsequence, $u_{\ve,g}^b$ converges  to the   map  $u_*$ that satisfies the harmonic map equation \cite{BBH94, HeWo08}
\begin{equation}
\label{eq:u* GL}
\left\{
\begin{aligned}
-\Delta u_*& = u_* |\nabla u_*|^2 \quad \text{on}\quad \Om\backslash\{a_1, \cdots , a_d\},\\
u_*&=g  \qon \pa\Omega, \\
|u_*|&=1 \quad\text{ on } ~\Om\backslash\{a_1, \cdots , a_d\}.
\end{aligned}
\right.
\end{equation}
Moreover, we have
\begin{equation}
\label{eq:harmonic-sing}
u_* (x) = \frac{x-a_1}{|x-a_1|} \cdots  \frac{x-a_d}{|x-a_d|} e^{i \psi(x)},
\end{equation}
where $\psi$ is harmonic on $\Om$  and  the   singularities $a_1, \cdots , a_d $ of $u_*$  minimizes the associated renormalized energy.
We refer to \cite{SaSe97} for the study of the Ginzburg-Landau  model  with and without a magnetic fields that describe superconductivity.

Now, we turn to the $n$-component  Ginzburg-Landau equations \eqref{eq:GL-n}.
If we assume $A_\mu=0$, then   \eqref{eq:GL-n} is reduced to
\begin{equation}
\label{eq:semilocal GL}
\left\{
\begin{aligned}
-\Delta u_i& = \frac{1}{\ve^2} u_i \Big(n-\sum_{j=1}^n |u_j|^2\Big) \qin \Omega,\\
u_i&=g_i  \qon \pa\Omega,
\end{aligned}
\right.
\end{equation}
for each $i=1,\cdots,n$.
One may regard \eqref{eq:semilocal GL} as a direct extension of  \eqref{eq:calssical GL} to $n$-component equations.
Here,  $g_1, \cdots, g_n : \pa \Om \to S^1 $ are  smooth maps such that
\begin{equation}
\label{eq:deg gj}
\deg (g_j) :=\deg (g_j,\pa\Om)=d_j   \in \nat\cup\{0\}  \qfor  j=1,\cdots, n.
\end{equation}
The system \eqref{eq:semilocal GL} is the Euler-Lagrange equations of  the   functional
\begin{equation}
\label{eq:ftnal E}
 E_{\ve,\Om} (u_1,\cdots,u_n) = \frac12 \int_\Om \sum_{j=1}^n |\nabla u_j|^2 dx + \frac{1}{4\ve^2} \int_\Om \Big(n-\sum_{j=1}^n |u_j|^2\Big)^2 dx
\end{equation}
for a pair of maps $(u_1,\cdots,u_n) \in H^1_{g_1} \times\cdots\times  H^1_{g_n}$.
For simplicity, we write $E_\ve^b$ and $E_\ve$ instead of $E_{\ve,\Om}^b$ and $E_{\ve,\Om}$ if there is no confusion on the domain $\Om$.
It is easy to check that
\begin{equation}
\label{eq:min prob}
\inf  \big\{ E_\ve (u_1,\cdots,u_n) : (u_1,\cdots,u_n) \in H^1_{g_1} \times\cdots\times H^1_{g_n} \big\}
\end{equation}
is achieved by some $(u_{1,\ve},\cdots,u_{n,\ve})  \in H^1_{g_1} \times  \cdots\times H^1_{g_n}$.
 We also denote by  $u_{j,\ve}^b$ a minimizer  for $E_\ve^b $ on $H^1_{g_j}$.

The purpose of this paper is to study the asymptotic behavior of minimizers $(u_{1,\ve},\cdots,u_{n,\ve})$ as $\ve\to0$.
In this study, there is a  remarkable difference between \eqref{eq:calssical GL} and \eqref{eq:semilocal GL}.
First, we consider the possible limit equation which could be a $n$-component generalization of \eqref{eq:u* GL} in some sense.
For \eqref{eq:calssical GL}, if $d\ne 0$, then the limit function $u_*$ in  \eqref{eq:harmonic-sing} has $d$ singularities.
However, even for the case $(d_1, \cdots, d_n)\ne (0,\cdots,0)$, the limit functions for \eqref{eq:semilocal GL}  turns out to  have no singularities.
This difference is related to the nonexistence of singular harmonic map in $H^1$.
In fact, we have $ H^1_g (\Om;S^1) = \emptyset $ for the single equation \eqref{eq:calssical GL}.
For the system  \eqref{eq:semilocal GL}, it is natural to define a function space analogously:
\begin{align*}
 &\calX(g_1,\cdots,g_n;\Om)\\
 =&  \Big\{  (u_1,\cdots,u_n) \in H^1_{g_1}(\Om; \mathbb C) \times\cdots\times H^1_{g_n} (\Om; \mathbb C) ~:~\sum_{j=1}^n |u_j |^2 = n~  \text{ a.e. on } \Om\Big\}.
\end{align*}
We write  $\calX(g_1,\cdots,g_n )$ if there is no confusion on domains.

The asymptotic behavior of minimizers for $E_\ve$ is closely related to the maps in $\calX(g_1,\cdots,g_n )$.
Since it is expected that $ |u_{1,\ve}|^2 +\cdots + |u_{n,\ve}|^2\to n$ as $\ve \to 0$, the limit functions, say, $( u_1^*\cdots,  u_n^*)$ will satisfy $ | u_1^*|^2 +\cdots +| u_n^*|^2= n$ a.e. on $\Om$.
Thus, it is important to know whether $\calX(g_1,\cdots,g_n )$ is nonempty or not.
We can see that if $\Psi=(u_1,\cdots,u_n) \in \calX(g_1,\cdots,g_n )$, then
\[\tilde{\Psi}=\Big( \frac{u_1}{\sqrt{n}}, \cdots,  \frac{u_n}{\sqrt{n}} \Big)~\in ~ S^{2n-1}.
\]
Since the homotopy group $\pi_1 (S^{2n+1})$ is trivial, there is no topological obstruction and
we expect that the limit problem will not have singularities.
More precisely, the next theorem tells us that $\calX(g_1,\cdots,g_n )$ is nonempty if $n\ge 2$.
%
%

\begin{theorem}\label{thm:X(g1,g2)-2}
If $d_j\ge 0$ for all $1\le j \le n$, then   $\calX(g_1,\cdots,g_n ) \ne \emptyset$.
Furthermore, if
\begin{equation}
\label{I(u,v)}
 I(u_1,\cdots,u_n) = \sum_{j=1}^n\int_\Om |\nabla u_j|^2  dx,
\end{equation}
 the minimization problem
\begin{equation}
\label{inf couple}
\begin{aligned}
\beta (g_1,\cdots, g_n) &=  \inf  \big\{ I(u_1,\cdots,u_n)    :  (u_1,\cdots,u_n) \in \calX(g_1,\cdots,g_n)\big\},
\end{aligned}
\end{equation}
is achieved in  $\calX(g_1,\cdots,g_n) $.
\end{theorem}
\begin{proof}
We choose smooth functions $v_j\in H_{g_j}^1(\Om;\cpx)$ such that $ \inf_\Om(|v_1|^2+\cdots+|v_n^2|)>0$ on $\Om$.
Set
\[w_j=\frac{\sqrt{n}v_j}{ \big(\sum_{k=1}^n|v_k|^2\big)^{1/2}}\quad\mbox{for each} \quad j=1,\cdots,n.
\]
Then,     $(w_1,\cdots, w_n) \in \calX(g_1,\cdots,g_n)  \ne \emptyset$.
Moreover,
\[ \int_\Om |\nabla w_j|^2~dx \le C  \int_\Om  \Big(  \frac{|\nabla v_j|^2}{  \sum_k|v_k|^2}  + \frac{|v_j|^2 (\sum_k \nabla |v_k|^2 )^2 }{ (\sum_k|v_k|^2)^3}  \Big) dx \le C.
\]
Thus, $\beta (g_1,\cdots, g_n)<\infty$ and it is achieved in $\calX(g_1,\cdots,g_n) $ by a standard variational argument.
\end{proof}

The associated Euler-Lagrange equations for \eqref{I(u,v)} are
\begin{equation}\label{eq:u1v1 system}
\left\{
\begin{aligned}
-\Delta u_j& = \frac{1}{n}u_j \sum_{k=1}^n |\nabla u_k|^2  \qin \Om, \quad u_j=g_j \qon \pa \Om,\\
& \sum_{j=1}^n|u_j|^2 =n \qin \Om.
\end{aligned}\right.
\end{equation}
Indeed, if $(u_1,\cdots,u_n)$ is a critical point of $I$, then we are led to
\[0=i'(0) = - \sum_{j=1}^n\int_\Om \Big[    \Delta u_j + \frac{1}{n}u_j \sum_{k=1}^n|\nabla u_k|^2  \Big]\psi_jdx
\]
by  using the fact $\nabla ( \sum_{j=1}^n|u_j |^2)=0$.
Here, we set $ i(t) = I \big(u_1(t),\cdots,u_n(t)\big)$ for $( \psi_1,\cdots,\psi_n) \in H^1_0 (\Om)\times\cdots\times H^1_0(\Om)$ and
\[ u_j(t) = \frac{n (u_j+t\psi_j)}{\sqrt{n \sum_{k=1}^n|u_k+t\psi_k|^2}}  ~\in~ \calX(g_1,\cdots,g_n).
\]
We will refer to \eqref{eq:u1v1 system} as $n$-component harmonic map equations.
The condition  $|u_*|=1$ in \eqref{eq:u* GL} makes its zeros singular.
Whereas the condition $|u_1|^2+\cdots+|u_n|^2=n$ in \eqref{eq:u1v1 system} do not force
  their zeros  singular since it does  not mean $|u_1|=\cdots=|u_n|=1$.
We also note from \eqref{eq:u1v1 system} that $u_1,\cdots, u_n$ share the same singularities if any.

Now, we are ready to state the first   main result of this paper.
\begin{theorem}\label{thm:conv}
Let $(u_{1,\ve},\cdots,u_{n,\ve})$ be a minimizer for \eqref{eq:min prob}.
Then, as $\ve \to 0$,  up to a subsequence,
\begin{align}
\label{eq:conv C1alpha}
   (u_{1,\ve},\cdots,u_{n,\ve}) \to  ( u_1^*,\cdots, u_n^*)  \qin C^{1,\alpha}(\bar\Om)\times\cdots\times C^{1,\alpha}(\bar\Om) .
\end{align}
Here, $( u_1^*,\cdots, u_n^*)$ is a minimizer of $I$ defined by \eqref{I(u,v)} and satisfies a system \eqref{eq:u1v1 system} of harmonic map equations.
Moreover, for any positive integer $k$,
\begin{align}
\label{eq:conv Ck}
   (u_{1,\ve},\cdots,u_{n,\ve}) \to  ( u_1^*,\cdots, u_n^*)  \qin C^k_{loc}(\Om)\times\cdots\times C^k_{loc}(\Om)   ,\\
   \label{eq:conv grad}
  \frac1{\ve^2} \Big( n-\sum_{j=1}^n |u_{j,\ve}|^2 \Big) \to \frac{1}{n} \sum_{j=1}^n |\nabla  u_j^*|^2   \qin C^k_{loc}(\Om)\times\cdots\times C^k_{loc}(\Om).
\end{align}
\end{theorem}

We will show that although  the degrees of the minimizers $(u_{1,\ve},\cdots,u_{n,\ve})$ are not $0$,   our energy $E_\ve (u_{1,\ve},\cdots,u_{n,\ve})$  is still bounded as $\ve \to 0$. 
This gives a different situation for the Ginzburg-Landau energy $E_\ve^b$ in  \eqref{eq:ftnal Eb}.
It is known that $E_\ve^b(u_{\ve,g}^b)$ grows logarithmically in the limit $\ve \to 0$ if $\deg (g,\pa \Om)\ne 0$.
See \cite{Sa98}  for optimal lower bounds for the energy   $E_\ve^b $ and many applications.

The convergence \eqref{eq:conv Ck} and \eqref{eq:conv grad} imply that $  u_1^*,\cdots, u_n^* $ are smooth.
The condition $| u_1^*|^2+ \cdots+ | u_n^*|^2 =n$ makes each $ u_j^*$ smooth at its zeros.
It is an interesting question whether $|u_j^*|=1$ on $\Omega$ for some $j$.
This is not true if $d_j>0$.
As the following theorem says, we are not sure that this can happen even for the case $d_j=0$.

\begin{theorem}\label{thm:properties of min seq}
Let $(u_{1,\ve},\cdots,u_{n,\ve})$ be a minimizing sequence for \eqref{eq:ftnal E}.
\begin{itemize}
\item[{\rm (i)}]
If $d_j>0$ for some $1\le j\le n$, then we have
\[\frac{1}{\ve^2}\int_\Om\big(1-|u_{j,\ve}|^2\big)^2dx\to\infty.
\]

\item[{\rm (ii)}]
Suppose that $d_j=0$ for all $j=1,\cdots,n$.
If $\al(g_1,\cdots,g_n) > \beta(g_1,\cdots,g_n)$, then   there is some $k\in\{1,\cdots,n\}$ such that
\[\frac{1}{\ve^2}\int_\Om\big(1-|u_{k,\ve}|^2\big)^2dx\to\infty.
\]
Here, $\al(g_1,\cdots,g_n)$ is defined by Remark \ref{rmk:alpha beta}.

\item[{\rm (iii)}]
For $n=2$, if either one of $d_1$ and $d_2$ is positive or $\al(g_1,g_2)>\beta(g_1,g_2)$, then
\[\frac{1}{\ve^2}\int_\Om\big(1-|u_{j,\ve}|^2\big)^2dx\to\infty, \quad\mbox{for  each $j=1,2$.}
\]
\end{itemize}
 \end{theorem}

This paper is organized as follows.
In Section 2 and   3, we   prove Theorem \ref{thm:conv}.
The basic idea is based on   \cite{BBH93}.
We deal with  interior and boundary estimates respectively  in Section 2 and 3.
We will focus on how the argument in  \cite{BBH93} for single equation \eqref{eq:calssical GL} can be generalized to the system \eqref{eq:semilocal GL} nontrivially.
In Section 4, we prove Theorem \ref{thm:properties of min seq}.
We   consider another  minimization problem for \eqref{I(u,v)} on a smaller space and see  how it is related to the question that $|u_j^*|=1$ for some $j$.
We also study some additional properties of solutions for \eqref{eq:semilocal GL} and \eqref{eq:u1v1 system}.

We list some notations and facts that are used hereafter.
The vector $\nu$ stands for the outward unit normal vector field on a given domain.
We write $B_r(x)$ or $B(x,r)$ to denote the ball of radius $r$ centered at a point $x$.
We often use the following fact: there exists $\ga_0 = \ga_0(\Om )>0$ such that
\begin{equation}\label{eq:ball cap Omega}
meas(\Om \cap B_r (x)) \geq \ga_0 r^2\quad\quad \forall x\in \Om,~ \forall r \leq 1.
\end{equation}

\section{Interior Estimates}\label{sec:conv}
\setcounter{equation}{0}

In this section, we will prove Theorem \ref{thm:conv}.
In what follows, we often use the following lemma.
\begin{lemma}[{\cite[Lemma A.1]{BBH93}}]
\label{lem:elliptic estimates}
Let $-\Delta u =f$ in  an open set  $\Omega \subset \rtwo$.
Given any $K\Subset \Omega$, we have
\[ \|\nabla u\|_{L^\infty(K)}^2 \le C_K\|u\|_{L^\infty(\Omega)} \Big( \|u\|_{L^\infty(\Omega)}+\|f\|_{L^\infty(\Omega)}\Big),
\]
where $C $ depends only on $\Omega $ and $K$.
In addition, if $u=0$ on $\pa \Om$, then
\[ \|\nabla u\|_{L^\infty(K)}^2 \le C \|u\|_{L^\infty(\Omega)}  \|f\|_{L^\infty(\Omega)},
\]
where $C $ depends only on $\Omega $.
\end{lemma}

The next lemma provides $L^\infty$-estimates  for   $u_{j,\ve}$ and their gradients for $j=1,\cdots,n$.

\begin{lemma}\label{lem:infty}
Let $(u_{1,\ve},\cdots,u_{n,\ve})$ be  any solution   of  the system \eqref{eq:semilocal GL}.
Then,  for $0 <\e <1$, we have
\begin{align}
\label{eq:infty}
&\sum_{j=1}^n |u_{j,\ve}|^2 < n   \quad \hbox{on} \quad \Om, \\
\label{eq:inftygradient}
&\|\nabla u_{j,\ve}\|_{L^\infty(\Om)} \leq \frac{C_0}{\ve} \qfor j=1,\cdots,n.
\end{align}
Here, $C_0$ is a constant that depends only on $g_j$ and $\Om$.
\end{lemma}
\begin{proof}
Let $f =  n- \sum_{j=1}^n|u_{j,\ve}|^2$.
Then,
\[   \Delta f \le \frac{2}{\ve^2} \Big( \sum_{j=1}^n|u_{j,\ve}|^2 \Big) f - 2 \sum_{j=1}^n  |\nabla u_{j,\ve}|^2 \le \frac{2}{\ve^2}  \Big( \sum_{j=1}^n|u_{j,\ve}|^2 \Big)  f  \qin \Om.
\]
Since $f=0$ on $\pa\Om$, we have $f>0$ in $\Om$ by the strong maximum principle.

To show   \eqref{eq:inftygradient}, let us decompose $u_{j,\ve}=u_{j,\ve}^1+u_{j,\ve}^2$
where
\[ \left\{
\begin{aligned}
-\Delta u_{j,\ve}^1& = \frac{1}{\ve^2} u_{j,\ve} \big(n-\sum_{j=1}^n|u_{j,\ve}|^2\big) \qin \Omega, \quad u_{j,\ve}^1=0 \qon \pa \Om,\\
-\Delta u_{j,\ve}^2& = 0 \qin \Omega, \quad u_{j,\ve}^2=g_j \qon \pa\Omega.
\end{aligned}
\right.\]
Then, $\| \nabla u_{j,\ve}^1\|_\infty \le C/\ve$ by Lemma \ref{lem:elliptic estimates} and $ \| \nabla u_{j,\ve}^2\|_\infty  \le C$ by elliptic estimates for $j=1,\cdots,n$.
\end{proof}


\begin{proposition}
\label{prop:conv in H^1}
Let $(u_{1,\ve},\cdots,u_{n,\ve}) $ be a minimizer for the problem \eqref{eq:min prob}.
Then,
\[ (u_{1,\ve},\cdots,u_{n,\ve}) \to  ( u_1^*,\cdots, u_n^*)  \qin   H^1_{g_1} \times  \cdots\times H^1_{g_n} \qas  \ve \to 0,
\]
where $( u_1^*,\cdots, u_n^*)$ is a minimizer for   \eqref{inf couple}.
Moreover,
\begin{equation}\label{potentiel d = 0}
\lim_{\ve \rightarrow 0}\frac{1}{ \ve^2} \int_\Om \Big(n-\sum_{j=1}^n|u_{j,\ve}|^2\Big)^2 dx = 0.
\end{equation}
\end{proposition}
\begin{proof}
Let $ (\tilu_1,\cdots,\tilu_n) \in  H^1_{g_1} \times  \cdots\times H^1_{g_n}$ be any minimizer for \eqref{inf couple}.
Since   $\sum_{j=1}^n|\tilu_j|^2=n$ a.e. on $\Om$ and   $E_\ve(u_{1,\ve},\cdots,u_{n,\ve}) \le E_\ve(\tilu_1,\cdots,\tilu_n)$, we have
\begin{equation}
\label{eq:energy upper bd d1d2=0}
\begin{aligned}
&  \frac12 \int_\Om \sum_{j=1}^n |\nabla u_{j,\ve}|^2 dx + \frac{1}{4\ve^2} \int_\Om \Big(n- \sum_{j=1}^n|u_{j,\ve}|^2\Big)^2 dx  \\
\le~& \frac12 \int_\Om  \sum_{j=1}^n|\nabla \tilu_j|^2  dx.
\end{aligned}
\end{equation}
Hence,
\begin{equation}
\label{eq:unif bound in H1 deg=0}
\text{$\{ (u_{1,\ve},\cdots,u_{n,\ve}) \}$ is uniformly bounded in $ H^1 \times\cdots\times H^1$}
\end{equation}
 and  there exists an $n$ tuple  $( u_1^*,\cdots, u_n^*) \in  H^1_{g_1} \times  \cdots\times H^1_{g_n}$ such that up to a subsequence,
\[ (u_{1,\ve},\cdots,u_{n,\ve}) \to ( u_1^*,\cdots, u_n^*)
\]
weakly in  $H^1 \times\cdots\times H^1$  and  strongly in  $L^p\times\cdots\times L^p$  for all $p\ge1$.
Moreover,  $\sum_{j=1}^n| u_j^*|^2 =n$ a.e. on $\Omega$ by \eqref{eq:energy upper bd d1d2=0} and thus $( u_1^*,\cdots, u_n^*) \in  \calX(g_1,\cdots,g_n)$.
Since $E_\ve(u_{1,\ve},\cdots,u_{n,\ve}) \le E_\ve( u_1^*,\cdots, u_n^*)$, we obtain
\[ \int_\Om \sum_{j=1}^n |\nabla u_{j,\ve}|^2  dx  \le  \int_\Om \sum_{j=1}^n |\nabla  u_j^*|^2  dx.
\]
This implies that
\begin{align*}
 &\int_\Om  \sum_{j=1}^n | \nabla u_{j,\ve} -\nabla  u_j^*|^2 dx \\
  \le~& 2  \int_\Om  \sum_{j=1}^n |\nabla  u_j^*|^2 dx -2 \int_\Om \sum_{j=1}^n  \nabla u_{j,\ve} \cdot \nabla  u_j^*dx ~\to ~ 0.
\end{align*}
So, $(u_{1,\ve},\cdots,u_{n,\ve}) \to  ( u_1^*,\cdots, u_n^*)$ strongly in $ H^1 \times\cdots\times H^1$.
Taking $\liminf_{\ve\to 0}$ on \eqref{eq:energy upper bd d1d2=0}, we are led to
\[      \int_\Om \sum_{j=1}^n |\nabla  u_j^*|^2 dx
\le   \int_\Om \sum_{j=1}^n|\nabla \tilu_j|^2  dx,
\]
and  we conclude that $( u_1^*,\cdots, u_n^*)$ is also a solution for the minimization problem \eqref{inf couple}.
Moreover, \eqref{potentiel d = 0} is a direct consequence of \eqref{eq:energy upper bd d1d2=0}.
\end{proof}

\begin{lemma}\label{lem:unif conv in L-infty deg=0}
$\sum_{j=1}^n |u_{j,\ve}|^2  \to n$ uniformly on $\overline\Om$.
\end{lemma}
\begin{proof}
Fix $x_0 \in \Om$.
By \eqref{eq:inftygradient}, we have that for $|x-x_0|<\ve \rho_\ve$,
\[ |u_{j,\ve}(x)|^2  \le  \big( |u_{j,\ve}(x_0)|+ C_0 \rho_\ve  \big)^2  \qfor j=1,\cdots,n.
\]
Since $|u_{j,\ve}|\le \sqrt{n}$ for all $j=1,\cdots,n$, we obtain
\begin{align*}
  n- \sum_{j=1}^n|u_{j,\ve}(x)|^2  &\ge n-    \sum_{j=1}^n|u_{j,\ve}(x_0)|^2- 2n\sqrt{n} C_0 \rho_\ve  -  nC_0^2\rho_\ve^2 \\
 & \ge \frac{1}{2}\Big(n-    \sum_{j=1}^n|u_{j,\ve}(x_0)|^2\Big),
\end{align*}
where the last inequality holds if and only if
\[ a_\ve (x_0) :=\frac{1}{2}\Big(n-    \sum_{j=1}^n|u_{j,\ve}(x_0)|^2\Big) \ge 2n\sqrt{n} C_0 \rho_\ve  + nC_0^2\rho_\ve^2,
\]
in other words,
\[ 0<\rho_\ve \le \frac{-n\sqrt{n}+ \sqrt{n^3+na_\ve(x_0) } }{nC_0}  = :b_\ve (x_0) .
\]
By taking $\rho_\ve=b_\ve $, we deduce from \eqref{eq:ball cap Omega} and   \eqref{potentiel d = 0} that
\begin{align*}
o(1)&= \frac{1}{ \ve^2} \int_{B_{\ve \rho_\ve}(x_0) \cap \Om}\Big(n-    \sum_{j=1}^n|u_{j,\ve}|^2\Big)^2 dx \ge  \ga_0 a_\ve^2 (x_0) b_\ve^2(x_0).
\end{align*}
Hence, $a_\ve (x_0) \to 0$ uniformly.
\end{proof}

\begin{lemma}\label{lem:loc H2 bound deg=0}
$\|(u_{1,\ve},\cdots,u_{n,\ve})\|_{H^2_{loc}  \times\cdots\times H^2_{loc}  } \le C$.
\end{lemma}
\begin{proof}
We note that
\[   \Delta \Big(\sum_{j=1}^n|\nabla u_{j,\ve}|^2\Big)  =~ 2\sum_{j=1}^n|D^2 u_{j,\ve}|^2  + 2\sum_{j=1}^n\nabla u_{j,\ve} \cdot \nabla (\Delta u_{j,\ve}).
\]
Hence, we have
\begin{align*}
 \sum_{j=1}^n|D^2 u_{j,\ve}|^2  =~ &  \frac12\sum_{j=1}^n \Delta \Big(|\nabla u_{j,\ve}|^2\Big) + \frac1{\ve^2} \sum_{j=1}^n|\nabla u_{j,\ve}|^2 \Big(n-\sum_{i=1}^n|u_{i,\ve}|^2\Big)\\
& \qquad - \frac{2}{\ve^2} \Big(\sum_{j=1}^nu_{j,\ve}\cdot \nabla u_{j,\ve}\Big)^2.
\end{align*}
Since $\sum_{j=1}^n|u_{j,\ve}|^2\to n$ uniformly on $\overline\Om$,  at least one of $u_{j,\ve}$ satisfies that $|u_{j,\ve}| \ge 1/2$ on $\Om$ for all small $\ve$.
If  $|  u_{k,\ve}| \ge 1/2$ for some $k\in\{1,\cdots,n\}$, then
\[ \frac1{\ve^2}   \Big(n-\sum_{i=1}^n|u_{i,\ve}|^2\Big) = \frac{|\Delta u_{k,\ve}|}{|u_{k,\ve}|} \le 2 |\Delta u_{k,\ve}|.
\]
Hence, we can write
\[ \sum_{j=1}^n|D^2 u_{j,\ve}|^2  \le~    \frac12\sum_{j=1}^n \Delta \Big(|\nabla u_{j,\ve}|^2\Big) + 2 \sum_{j=1}^n|\nabla u_{j,\ve}|^2   |\Delta u_{k,\ve}|,
\]
which implies that
\begin{equation}
\label{eq:D2u ineq}
\frac12\sum_{j=1}^n|D^2 u_{j,\ve}|^2  \le  \frac12 \sum_{j=1}^n \Delta \Big(|\nabla u_{j,\ve}|^2\Big)+ C\sum_{j=1}^n|\nabla u_{j,\ve}|^4.
\end{equation}

Fix $x_0$ and $r = \dist(x_0,\pa \Om)/4$.
Let $\zeta$ be a smooth function such that $\zeta=1$ on $B_r(x_0)$, $\supp~ \zeta \subset B_{2r}(x_0)$ and $0\le \zeta \le 1$.
Multiplying \eqref{eq:D2u ineq} by $\zeta^2$ and using the Sobolev embedding $W^{1,1}(\Om)\hookrightarrow L^2(\Om)$, we are led to
\begin{align*}
& \frac12 \int_{B_{2r}(x_0) } \zeta^2\sum_{j=1}^n|D^2 u_{j,\ve}|^2 dx\\
\le ~& C  \int_{B_{2r}(x_0) } \sum_{j=1}^n|\nabla u_{j,\ve}|^2dx + C \int_{B_{2r}(x_0) }  \sum_{j=1}^n \big(\zeta|\nabla u_{j,\ve}|^2\big)^2dx\\
\le ~& C  \int_{B_{2r}(x_0) }  \sum_{j=1}^n|\nabla u_{j,\ve}|^2dx + C \int_{B_{2r}(x_0) } \zeta\sum_{j=1}^n |\nabla u_{j,\ve}| \, |D^2 u_{j,\ve}|dx.
\end{align*}
In the sequel, by \eqref{eq:unif bound in H1 deg=0}
\[ \frac12 \int_{B_{2r}(x_0) } \zeta^2\sum_{j=1}^n|D^2 u_{j,\ve}|^2dx \le C  \int_{\Om } \sum_{j=1}^n|\nabla u_{j,\ve}|^2dx  \le C.
\]
Now, the standard covering argument shows that each $u_{j,\ve}$ is  uniformly bounded in $H^2_{loc}(\Om)$.
\end{proof}

Let
\begin{equation} \label{eq:f-epsilon}
f_\ve= \frac{1}{\ve^2}\Big(n-\sum_{j=1}^n|u_{j,\ve}|^2\Big) .
\end{equation}
Then, we can rewrite \eqref{eq:semilocal GL} as
\begin{equation} \label{eq:u,v eqn via f-epsilon}
\left\{ \begin{aligned}
-\Delta u_{j,\ve} &= f_\ve u_{j,\ve}  \qin \Om,\\
u_{j,\ve} &=g_j \qon \pa \Om.
\end{aligned} \right.
\end{equation}
A simple calculation gives
\begin{equation} \label{eq:f-epsilon eqn}
\left\{ \begin{aligned}
  -\ve^2 \Delta f_\ve + 2 \sum_{j=1}^n|u_{j,\ve}|^2 f_\ve &= 2\sum_{j=1}^n|\nabla u_{j,\ve}|^2 \qin \Om,\\
f_\ve& =0 \qon \pa \Om .
\end{aligned} \right.
\end{equation}
Given a compact set  $K \subset \Om$, we define
\[ \calA_K =\Big\{ \tilde{K} \subset \Om~|~ K\Subset \tilde{K} \Subset \Om \Big\}.
\]
For   $\tilde{K} \in  \calA_K $, we define a set  $\calF(K,\tilde{K})$ of smooth functions  by
$$
\calF(K,\tilde{K})=\Big\{ \vp\in C^\infty(\Om) ~|~ 0\le\vp\le 1, \vp\equiv 1  ~\mbox{on $K$ and  $\supp~\vp \subset \tilde{K}$}\Big\}.
$$
In what follows, $C_K$, or simply $C$,  denotes a generic constant depending on a compact set $K$ but independent of $\ve$.
For simplicity, we also often write $\| (u_1,\cdots,u_n)\|_{X \times\cdots\times X}$ as $\| (u_1,\cdots,u_n)\|_{X }$ for a function space $X$.

\begin{lemma}\label{lem:f,h C1 bound deg=0}
For any $p \ge 1$, we have
$\|(u_{1,\ve},\cdots,u_{n,\ve})\|_{W^{2,p}_{loc} \times\cdots\times  W^{2,p}_{loc}} \le C$ and $\|f_\ve\|_{C^0_{loc} } \le C$.
\end{lemma}
\begin{proof}
By Lemma \ref{lem:unif conv in L-infty deg=0}, we may assume that $\sum_{j=1}^n|u_{j,\ve}|^2 \ge 1$ on $\Om$.
Given $K\Subset \Om$, we denote $K=K_0$ and choose compact sets $K_j \in \calA_{K_{j-1}}$ for $j=1,2,3$ and $\vp  \in \calF(K_3,K_2)$.
Set $F_\ve = \vp  f_\ve$.
Then,
\begin{align*}
-\ve^2 \Delta F_\ve &= -\ve^2 f_\ve \Delta \vp - 2 \ve^2 \nabla \vp \cdot \nabla f_\ve - \ve^2 \vp \Delta f_\ve\\
&= -\Big(n-\sum_{j=1}^n|u_{j,\ve}|^2\Big) \Delta \vp  +4 \sum_{j=1}^nu_{j,\ve} (\nabla \vp \cdot \nabla u_{j,\ve})\\
&\qquad  -2F_\ve \sum_{j=1}^n|u_{j,\ve}|^2 + 2 \vp \sum_{j=1}^n|\nabla u_{j,\ve}|^2 \\
&\le -2F_\ve + C_K\Big( 1+ \sum_{j=1}^n|\nabla u_{j,\ve}|^2\Big).
\end{align*}
Hence,
\begin{equation}
\label{eq:vp f-ve ineq}
-\ve^2 \Delta F_\ve  +2F_\ve  \le C_K\Big( 1+ \sum_{j=1}^n|\nabla u_{j,\ve}|^2\Big) \qon K_3.
\end{equation}
For $p>2$, multiplying \eqref{eq:vp f-ve ineq} by $F_\ve^{p-1} $, we are led to
\begin{align*}
& \ve^2 (p-1) \int_{K_3} |\nabla F_\ve|^2   F_\ve^{p-2}dx + 2 \int_{K_3} F_\ve^pdx \\
 \le~&C_K \Big( \int_{K_3}\Big( 1+ \sum_{j=1}^n|\nabla u_{j,\ve}|^2\Big) ^pdx \Big)^{\frac{1}{p} } \Big( \int_{K_3}  F_\ve^p dx\Big)^{\frac{p-1}{p} } \\
 \le ~& C_K \| (u_{1,\ve},\cdots,u_{n,\ve})\|_{H^2 (K_3)} \| F_\ve \|_{L^p(K_3)}^{p-1} ~\le ~ C_K \| F_\ve \|_{L^p(K_3)}^{p-1}.
\end{align*}
Here, we used   Lemma \ref{lem:loc H2 bound deg=0} in the last inequality.
Then, we deduce from Young's inequality that $\| F_\ve \|_{L^p({K_3})} \le C_K$.
This implies by \eqref{eq:u,v eqn via f-epsilon} that $\| \Delta u_{j,\ve}\|_{L^p (K_2)} \le C_K$  for any $p>2$ and $j=1,\cdots,n$.
So,
\[ \| (u_{1,\ve},\cdots,u_{n,\ve}) \|_{W^{2,p}(K_1)}\le  C_K
\]
 by the interior regularity.
 Moreover, $\|   (u_{1,\ve},\cdots,u_{n,\ve}) \|_{C^1(K_1)} \le C_K$ by the Sobolev embedding.

Now choose $\tilphi \in \calF(K,K_1)$ and set $\tilF_\ve =\tilphi f_\ve$.
Then, as above, we have
\[  -\ve^2 \Delta \tilF_\ve  +2\tilF_\ve  \le C_K\Big( 1+ \sum_{j=1}^n|\nabla u_{j,\ve}|^2\Big) \qon K_1.
\]
By applying the maximum principle to this inequality, we are led to
\[ \max_{K_1} \tilF_\ve \le C_K\Big( 1+ \sum_{j=1}^n \|\nabla u_{j,\ve}\|_{L^\infty(K_1)}^2\Big)\le C_K.
\]
In particular, we conclude that $\| f_\ve\|_{L^\infty(K)} \le C_K$.
\end{proof}

\begin{lemma}\label{lem:f,h Ck bound deg=0}
Given a nonnegative integer $k$ and a real number $p \ge 1$, we have
\begin{align}
\label{eq:u estimate k-step}
& \|(u_{1,\ve},\cdots,u_{n,\ve})\|_{W^{k+2,p}_{loc} \times\cdots\times W^{k+2,p}_{loc}} \le C  \qand      \| f_\ve\|_{C^k_{loc}   }\le C.
 \end{align}
\end{lemma}
\begin{proof}
We use an induction on $k$.
The case $k=0$ follows from Lemma \ref{lem:f,h C1 bound deg=0}.
Suppose that \eqref{eq:u estimate k-step} is true for $k$.
By the Sobolev embedding, $\| (u_{1,\ve},\cdots,u_{n,\ve})\|_{C^{k+1}_{loc}} \le C$.
By   the induction hypothesis, we note from  \eqref{eq:f-epsilon eqn} that
\[   -\ve^2 \Delta \pa^k f_\ve =   \pa^k \Big(- 2 \sum_{j=1}^n|u_{j,\ve}|^2 f_\ve+ 2\sum_{j=1}^n|\nabla u_{j,\ve}|^2 \Big) =O(1)
\]
 on every compact set.
Then,  by applying   Lemma \ref{lem:elliptic estimates}, we obtain that
\begin{equation}
\label{eq:partial k+1 f}
\|\partial^{k+1} f_\ve\|_{L^\infty_{loc} } \le C \ve^{-1}.
\end{equation}
By \eqref{eq:u,v eqn via f-epsilon} and \eqref{eq:partial k+1 f},  on every compact set,
\[ -\Delta \pa^{k+1} u_{j,\ve} = \pa^{k+1} (f_\ve u_{j,\ve}) =O(\ve^{-1}).
\]
Hence, by  Lemma \ref{lem:elliptic estimates} again,
\begin{equation}
\label{eq:partial k+2 u}
\|\partial^{k+2} u_{j,\ve}\|_{L^\infty_{loc} } \le C \ve^{-1/2}.
\end{equation}
By \eqref{eq:partial k+1 f} and \eqref{eq:partial k+2 u}, we get
\[   -\ve^2 \Delta \pa^{k+1} f_\ve =   \pa^{k+1} \Big(- 2 \sum_{j=1}^n|u_{j,\ve}|^2 f_\ve+ 2\sum_{j=1}^n|\nabla u_{j,\ve}|^2 \Big) =O(\ve^{-1})
\]
 on every compact set.
Applying Lemma \ref{lem:elliptic estimates} to this equation, we are led to
\begin{equation}
\label{eq:partial k+2 f}
\|\partial^{k+2} f_\ve\|_{L^\infty_{loc} } \le C \ve^{-2}.
\end{equation}

Now, given $K\Subset \Om$, we denote $K=K_0$ and choose compact sets $K_j \in \calA_{K_{j-1}}$ for $j=1,2,3$ and $\vp  \in \calF(K_2,K_3)$.
Set $F_\ve = \vp  \pa^{k+1} f_\ve$.
It follows from \eqref{eq:partial k+1 f} and \eqref{eq:partial k+2 f} that
\begin{align*}
-\ve^2 \Delta F_\ve &= -\ve^2   \pa^{k+1} f_\ve \Delta \vp - 2 \ve^2 \nabla \vp \cdot   \nabla  \pa^{k+1} f_\ve - \ve^2 \vp \Delta   \pa^{k+1} f_\ve\\
&= O(1)  - \ve^2 \vp \Delta   \pa^{k+1} f_\ve.
\end{align*}
On the other hand, by \eqref{eq:f-epsilon eqn}
\begin{align*}
 - \ve^2 \vp \Delta   \pa^{k+1} f_\ve &= - 2  \sum_{j=1}^n|u_{j,\ve}|^2F_\ve  - 2\vp \sum_{i+l = k+1, ~l<k+1}\pa^{i} \Big(\sum_{j=1}^n|u_{j,\ve} |^2\Big) \pa^l f_\ve \\
 & \qquad + 2 \vp \pa^{k+1} \Big(\sum_{j=1}^n|\nabla u_{j,\ve} |^2\Big)\\
&\le -2 F_\ve + C_K \Big( 1+ \sum_{j=1}^n|D^{k+2} u_{j,\ve}|\Big).
\end{align*}
Hence,
\begin{equation}
\label{eq:vp f-ve ineq k+1}
-\ve^2 \Delta F_\ve  +2F_\ve  \le C_K\Big( 1+ \sum_{j=1}^n|D^{k+2} u_{j,\ve}|\Big) \qon K_3.
\end{equation}
For $p>2$, multiplying \eqref{eq:vp f-ve ineq k+1} by $F_\ve^{p-1} $, we are led to
\begin{align*}
& \ve^2 (p-1) \int_{K_3} |\nabla F_\ve|^2   F_\ve^{p-2}dx + 2 \int_{K_3} F_\ve^p dx\\
 \le~&C_K \Big( \int_{K_3}\Big( 1+ \sum_{j=1}^n|D^{k+2} u_{j,\ve}|\Big)^p  dx\Big)^{\frac{1}{p} } \Big( \int_{K_3}  F_\ve^p dx\Big)^{\frac{p-1}{p} } \\
 \le ~& C_K \| (u_{1,\ve},\cdots,u_{n,\ve})\|_{W^{k+2,p} (K_3)} \| F_\ve \|_{L^p(K_3)}^{p-1} ~\le ~ C_K \| F_\ve \|_{L^p(K_3)}^{p-1},
\end{align*}
which implies that   $\| F_\ve \|_{L^p({K_3})} \le C_K$.
As a consequence, by \eqref{eq:u,v eqn via f-epsilon}
\[ \| \Delta \pa^{k+1} u_{j,\ve}\|_{L^p (K_2)} \le C_K \qfor j=1,\cdots,n.
\]
So,
\begin{equation}
\label{eq:induction k+1 u}
 \| (u_{1,\ve},\cdots,u_{n,\ve}) \|_{W^{k+3,p}(K_1)  }\le  C_K.
\end{equation}
In particular,   $\| \pa^{k+2} ( u_{1,\ve},\cdots,u_{n,\ve}) \|_{C^0(K_1)} \le C_K$.

Now choose $\tilphi \in \calF(K,K_1)$ and set $\tilF_\ve =\tilphi \pa^{k+1} f_\ve$.
As above,  we obtain
\[  -\ve^2  \Delta \tilF_\ve  +2\tilF_\ve  \le C_K\Big( 1+ \sum_{j=1}^n|D^{k+2} u_{j,\ve}|\Big) \qon K_1.
\]
Then, we deduce from the maximum principle that
\[ \max_{K_1} \tilF_\ve \le C_K\Big(1+ \sum_{j=1}^n\|D^{k+2} u_{j,\ve}\|_{L^\infty(K_1)}\Big)\le C_K.
\]
In the sequel, we conclude that
\begin{equation}
\label{eq:induction k+1 f}
\|\pa^{k+1}  f_\ve\|_{L^\infty(K)} \le C_K.
\end{equation}
Therefore, \eqref{eq:u estimate k-step} is true for $k+1$ by \eqref{eq:induction k+1 u} and \eqref{eq:induction k+1 f}.
This completes the proof.
\end{proof}

\section{Boundary Estimates and Proof of  Theorem  \ref{thm:conv} }\label{sec:Boundary Estimates}
\setcounter{equation}{0}

\begin{lemma}
\label{lem:pohozaev}
For any solution of problem \eqref{eq:semilocal GL} there is a constant $C$ such that  for all $\ve>0$,
\begin{equation}
\label{eq:pohozaev}
\int_{\partial\Om}  \sum_{j=1}^n\Big|\frac{\partial u_{j,\e}}{\partial\nu} \Big|^2d\sigma\leq C.
\end{equation}
Here, $\nu$ is the outward unit normal vector field on $\pa \Om$.
\end{lemma}
\begin{proof}
For simplicity, we drop the subscript $\ve$.
Let us  multiply   \eqref{eq:semilocal GL} by $V \cdot \nabla u_{i}$,
where $V=(V_1,V_2)$ is a smooth vector field on $\Om$ such that $V=\nu$ on $\partial \Om$.
By integrating them by part,  we obtain
\[ \text{(LHS)}=\int_\Om \Delta u_i(V\cdot \nabla u_i)~dx= \int_{\partial\Om}\Big|\frac{\partial u_i}{\partial\nu}\Big|^2d\sigma-\frac{1}{2}\int_{\partial\Om} |\nabla u_i|^2d\sigma +O(1)
\]
since $\nabla u_i$ is uniformly  bounded in $L^2(\Om)$.
On the other hand,
\[\text{(RHS)} =\frac{1}{2\ve^2}\int_\Om \Big(n -\sum_{j=1}^n|u_{j}|^2\Big)\big(V\cdot\nabla |u_i|^2\big)dx.
\]
Adding these for each $i=1,\cdots,n$, we have
\begin{align*}
&\sum_{i=1}^n\Big(\int_{\partial\Om}\Big|\frac{\partial u_i}{\partial\nu}\Big|^2d\sigma-\frac{1}{2}\int_{\partial\Om} |\nabla u_i|^2d\sigma\Big)\\
=&~ \frac{1}{2\ve^2}\int_\Om \Big(n -\sum_{j=1}^n|u_{j}|^2\Big)\sum_{i=1}^n\big(V\cdot\nabla |u_i|^2\big)dx+O(1) \\
=& ~\frac{1}{4\ve^2}\int_\Om \Big(n -\sum_{j=1}^n|u_{j}|^2\Big)^2 (\nabla \cdot  V)dx +O(1) =O(1),
\end{align*}
where the last inequality comes from \eqref{eq:energy upper bd d1d2=0}.
Combining these identities, we conclude that
\[\sum_{i=1}^n \int_{\partial\Om}\Big|\frac{\partial u_i}{\partial\nu}\Big|^2d\sigma
= \sum_{i=1}^n  \int_{\partial\Om} \Big|\frac{\partial g_i}{\partial\tau}\Big|^2d\sigma+O(1)=O(1).
\]
Here, $\tau$ is the tangential vector field on $\pa \Om$.
\end{proof}

In this section, the estimates of solutions for problem \eqref{eq:semilocal GL} would be proved up to the boundary.

\begin{lemma}\label{lem:H2 bound}
$\|(u_{1,\ve},\cdots,u_{n,\ve})\|_{H^2 (\Om) \times\cdots\times H^2  (\Om) } \le C$.
\end{lemma}
\begin{proof}
In view of Lemma \ref{lem:loc H2 bound deg=0}, it suffices to prove uniform $H^2$-boundedness near each $x_0\in\partial\Om$.
We may assume $x_0=0$ and change local coordinates $(x_1,x_2)\to(y_1,y_2)=(x_1,x_2-h(x_1))$
where  $h$ represent $\partial\Om$ locally and   $h'(0)=0$.
We set $\tilu_{i,\ve}(y_1,y_2)=u_{i,\ve}(x_1,x_2)$ and $\tilde{g}_i(y_1,y_2)=g_i(x_1,x_2)$ on $U=\big\{(y_1,y_2) : y_2>0\big\}\cap B_r(0)$ for some $r>0$, and rewrite  \eqref{eq:semilocal GL} as
\[ \left\{
\begin{aligned}
-L\tilu_{i,\ve} &=\frac{1}{\ve^2}\tilu_{i,\ve}\Big(n-\sum_{j=1}^n|\tilu_{j,\ve}|^2\Big)  \qon  U,\\
\tilu_{i,\ve}& =\tilde{g}_{i,\ve} \qon \{y_2=0\}\cap\partial U.
\end{aligned}
\right.\]
Here,
\begin{align*}
L=\sum_{k,l=1}^2\frac{\partial}{\partial y_l}\Big(a_{kl}\frac{\partial}{\partial y_k}\Big) \mbox{ with $a_{11}=1$, $a_{12}=a_{21}=-h'$ and $a_{22}=1+(h')^2$}
\end{align*}
is a strongly elliptic operator if $r$ is small enough.
For simplicity, we write $\tilu_{j,\ve}$ as $u_j$ in what follows.

Let $A= \frac12\sum_{j=1}^n |\nabla u_j|^2$.
By  a direct calculation,
\begin{align*}
L [A] & =   \sum_{j=1}^n  \sum_{k,l,r=1}^2 \Big(a_{kl} (u_j)_{y_ky_r} (u_j)_{y_ly_r} + (u_j)_{y_r} \cdot L \big[ (u_j)_{y_r}\big]\Big) \\
& \ge \alpha \sum_{j=1}^n |D^2 u_j|^2 +\sum_{j=1}^n \sum_{r=1}^2 (u_j)_{y_r} \cdot L \big[ (u_j)_{y_r}\big],
\end{align*}
where $\al$ is the ellipticity constant of $L$.
We note that
\begin{align*}
&\sum_{j=1}^n \sum_{r=1}^2  (u_j)_{y_r} \cdot L \big[ (u_j)_{y_r}\big] \\
=~&   \sum_{j=1}^n \sum_{k,l,r=1}^2 (u_j)_{y_r} \cdot  \Big\{  \big(L[u_j]\big)_{y_r} - \big[ (a_{kl})_{y_r} (u_j)_{y_k} \big]_{y_l} \Big\}\\
=~&- \frac{1}{\ve^2}\sum_{j=1}^n |\nabla u_j|^2 \Big( n - \sum_{i=1}^n |u_i|^2\Big) + \frac{2}{\ve^2}  \sum_{r=1}^2 \Big(\sum_{j=1}^n u_j \cdot (u_j)_{y_r}\Big) \Big(\sum_{i=1}^n u_i \cdot (u_i)_{y_r} \Big) \\
& -\sum_{j=1}^n \sum_{k,l,r=1}^2   (u_j)_{y_r}  \big[ (a_{kl})_{y_ly_r} (u_j)_{y_k}+ (a_{kl})_{y_r} (u_j)_{y_ky_l} \big]  \\
\ge~& - \frac{1}{\ve^2}\sum_{j=1}^n |\nabla u_j|^2 \Big( n - \sum_{i=1}^n |u_i|^2\Big) - C\sum_{j=1}^n (|\nabla u_j|^2+ |D^2u_j|^2).
\end{align*}
By Lemma \ref{lem:unif conv in L-infty deg=0}, given $y \in U$, if $\ve$ is small, then we can find  $k=k(y)$ such that $|u_k(y)|\ge 1/2$.
Then,
\[ |L[u_k](y)| = \frac{1}{\ve^2} |u_k(y)| \Big( n - \sum_{j=1}^n |u_j(y)|^2 \Big) \ge  \frac{1}{2\ve^2}\Big( n - \sum_{j=1}^n |u_j(y)|^2 \Big).
\]
Hence,
\begin{align*}
 \frac{1}{ \ve^2}\Big( n - \sum_{j=1}^n |u_j(y)|^2 \Big) & \le 2|L[u_k]  (y) | =2 \Big| \Big( a_{kl}(y) (u_k)_{y_k}(y) \Big)_{y_l} \Big|  \\
 & \le C  \sum_{j=1}^n  ( |\nabla u_j(y)| + |D^2u_j(y)|).
\end{align*}
As a consequence, by Young's inequality,  we are led to
\[-L[A]+\frac{\alpha}{2}\sum_{j=1}^n |D^2 u_j|^2  \le C\Big(1 + \sum_{j=1}^n |\nabla u_j|^4  \Big).
\]

Let us choose $\zeta \in \calF(B_{r/2}(0),B_r(0))$.
Then,
\begin{equation}
\label{eq:H2 bdy ineq}
 \frac{\alpha}{2}\sum_{j=1}^n \int_U \zeta^2  |D^2 u_j|^2dx  \le C+ C \sum_{j=1}^n\int_U  \zeta^2  |\nabla u_j|^4 dx+ \int_U  \zeta^2 L[A]dx.
\end{equation}
By integrating by parts, we obtain
\begin{align*}
 \int_U  \zeta^2 L[A]dx&  =  \int_U  \zeta^2 \big[  a_{kl} A_{y_l} \big]_{y_k}dx\\
 & = \int_U AL[\zeta^2] dx+ 2 \int_{ \{y_2=0\} }   a_{12}A (\zeta^2)_{y_1} +   \int_{ \{y_2=0\} }   a_{22}A (\zeta^2)_{y_2}\\
 & \quad +   \int_{ \{y_2=0\} } \zeta^2 A ( a_{12} )_{y_1}   -     \int_{ \{y_2=0\} }  \zeta^2  a_{22}A_{y_2}  .
\end{align*}
The first four terms are uniformly bounded by Lemma \ref{lem:pohozaev}.
Furthermore,
\begin{align*}
&\int_{ \{y_2=0\} }  \zeta^2  a_{22}A_{y_2}\\
=~&  \sum_{j=1}^n  \Big(  \int_{ \{y_2=0\} }   \zeta^2  a_{22}   (u_j)_{y_1} \cdot (u_j)_{y_1y_2} +
   \zeta^2  a_{22}  (u_j)_{y_2} \cdot (u_j)_{y_2y_2} \Big) \\
 =~& \sum_{j=1}^n \big( (I)+(II) \big).
\end{align*}
By Lemma \ref{lem:pohozaev},
\[ (I)=-  \int_{ \{y_2=0\} } \Big(   (\zeta^2  a_{22})_{y_1}  (u_j)_{y_1} \cdot (u_j)_{ y_2} +  \zeta^2  a_{22}  (g_j)_{y_1y_1}  \cdot (u_j)_{ y_2}\Big)=O(1).
\]
Since  $Lu_j=0$  on $\{y_2=0\}$ and $a_{11}=1$, we have
\begin{align*}
(II) &=-  \int_{ \{y_2=0\} }  \zeta^2 \Big( (u_j)_{y_2} \cdot (g_j)_{y_1y_1}   +  (a_{12})_{y_1}|(u_j)_{y_2}|^2 \\
& \qquad +~(a_{21})_{y_2} (u_j)_{y_2} \cdot (g_j)_{y_1}+  (a_{22})_{y_2}  | (u_j)_{y_2}|^2\Big)  - \int_{ \{y_2=0\} }  \zeta^2 a_{12} ( | (u_j)_{y_2}|^2)_{y_1}\\
& = O(1) + \int_{ \{y_2=0\} }  (\zeta^2 a_{12})_{y_1}  | (u_j)_{y_2}|^2 ~=~ O(1).
\end{align*}
Here, we also used Lemma \ref{lem:pohozaev}.
In the sequel, we can rewrite \eqref{eq:H2 bdy ineq} as
\[  \frac{\alpha}{2}\sum_{j=1}^n \int_U \zeta^2  |D^2 u_j|^2dx  \le C+ C \sum_{j=1}^n\int_U  \zeta^2  |\nabla u_j|^4dx .
\]
Now, using this inequality and employing the same argument of the proof of Lemma \ref{lem:loc H2 bound deg=0}, we can show
\[   \sum_{j=1}^n \int_{B_{r/2}(x_0)}  |D^2 u_j|^2dx  \le C.
\]
This completes the proof.
\end{proof}

\begin{proposition}\label{prop:f-ep on omega}
For any $p \ge 1$, we have
\[  \|(u_{1,\ve},\cdots,u_{n,\ve})\|_{W^{2,p} (\Om) \times\cdots\times W^{2,p}(\Om) } \le C_p
\]
Moreover, if   $f_\ve$ is defined by \eqref{eq:f-epsilon}, then
$ \|f_\ve\|_{C^0(\Omega)}\le C$.
\end{proposition}
\begin{proof}
By Lemma \ref{lem:H2 bound} and the Sobolev embedding,
\[ \|(u_{1,\ve},\cdots,u_{n,\ve})\|_{W^{1,p} (\Om) \times\cdots\times W^{1,p}  (\Om) } \le C \quad \forall p \ge 1.
\]
We keep the notation \eqref{eq:f-epsilon}.
By   Lemma \ref{lem:unif conv in L-infty deg=0} and  \eqref{eq:f-epsilon eqn}, we obtain
\begin{equation}
\label{eq:f-ep global}
-2\ve^2\Delta f_\ve+f_\ve\le 4\sum_{j=1}^n|\nabla u_{j,\ve}|^2\quad\mbox{on $\Om$}.
\end{equation}
For $q>1$, since $f_\ve=0$ on $\partial\Om$, multiplying this equation by $f_\ve^{q-1}$, we have
\[\int_\Om f_\ve^qdx\le 4\sum_{j=1}^n \int_\Om  |\nabla u_{j,\ve}|^2f_\ve^{q-1}dx\le 4\sum_{j=1}^n\|\nabla u_{j,\ve}\|^2_{2q}\|f_\ve\|^{q-1}_q,
\]
which implies that $\|f_\ve\|_q\le C$ for some $C=C(n,q)>0$.
Hence,  each $\Delta u_{j,\ve}$ is uniformly bounded in $L^q(\Om)$ for any $q \ge 1$.
By applying the elliptic  regularity  to \eqref{eq:u,v eqn via f-epsilon}, we obtain that
\[ \|(u_{1,\ve},\cdots,u_{n,\ve})\|_{C^{1,\alpha} (\Om) \times\cdots\times C^{1,\alpha}  (\Om) } \le C \quad \forall \al \in (0,1).
\]
By the maximum principle,   we obtain from \eqref{eq:f-ep global} that $ \|f_\ve\|_{C^0(\Omega)}\le C$.
In particular, the right hand side of \eqref{eq:u,v eqn via f-epsilon} is uniformly bounded in $C^0(\Om)$.
Hence, each $u_{j,\ve}$ is also uniformly bounded in $W^{2,p}(\Om)$ for any $p\ge 1$.
\end{proof}


{\bf Proof of Theorem \ref{thm:conv}.}
The proof of \eqref{eq:conv Ck} follows readily from Proposition \ref{prop:conv in H^1} and Proposition \ref{prop:f-ep on omega}.
To show \eqref{eq:conv grad}, let $f_\ve$ be defined by \eqref{eq:f-epsilon} and set
\[ h_\ve=\sum_{j=1}^n |u_{j,\ve}|^2  \qand f_* = \frac{1}{n} \sum_{j=1}^n   |\nabla u_j^*|^2 .
\]
Then, $-\Delta u_j^* = u_j^* f_*$ and $\|h_\ve - n\|_{C^k_{loc}(\Om)} = \ve^2 \|f_\ve\|_{C^k_{loc}(\Om)} \to 0$.
We note that
\[ - \sum_{j=1}^n (u_{j,\ve}  \Delta u_{j,\ve}  - u_j^*\Delta u_j^*) = n(f_\ve-f_*) + f_\ve (h_\ve - n).
\]
Hence, as $\ve \to 0$,
\begin{align*}
& \|f_\ve -f_*\|_{C^k_{loc}(\Om)} \\
\le~&  \frac{1}{n} \Big( \| f_\ve (h_\ve -n)\|_{C^k_{loc}(\Om)} + \sum_{j=1}^n \| u_{j,\ve} \Delta u_{j,\ve} - u_j^*\Delta u_j^* \|_{C^k_{loc}(\Om)} \Big) \to 0.
\end{align*}
This completes the proof.
\qed


\section{Further Properties of Solutions}\label{sec:Properties}
\setcounter{equation}{0}

In   this section, we study some properties of solutions of the $n$-component Ginzburg-Landau equations \eqref{eq:semilocal GL} and the generalized harmonic map equations \eqref{eq:u1v1 system}.
First, the next proposition tells us that if $g_j$ is a rotation of $g_1$ for each $j$ and $\ve$ is not so small, then $u_j$ is a rotation of $u_1$ for any solution pair  $(u_1, \cdots, u_n)$ of \eqref{eq:semilocal GL}.
Consequently, each of   $u_j$ and $u_k$ is a rotation of the other.

\begin{proposition}\label{prop:soln property 1}
Let $(u_1, \cdots, u_n)$ be a solution of \eqref{eq:semilocal GL}.
Assume that for each $j $, there exists $\ga_j \in \rone$ such that $g_j=e^{i\gamma_j} g_1$.
Then,  $u_j=e^{i\gamma_j}  u_1$ for all $\ve>\sqrt{n/\la_1} $.
Here, $\la_1$ is the first eigenvalue of $ -\Delta$ on $\Om$ with the Dirichlet boundary condition.
\end{proposition}
\begin{proof}
Set $\tilu_j=e^{-i\ga_j}u_j$.
Then, $\tilu_j$ satisfies
\begin{align*}
-\Delta \tilu_j& = \frac{1}{\ve^2} \tilu_j \Big(n-\sum_{k=1}^n |\tilu_k|^2\Big) \qin \Omega,\\
\tilu_j&=g_1  \qon \pa\Omega.
\end{align*}
If $w_j=\tilu_j-u_1$, then
\begin{align*}
-\Delta w_j + \frac{1}{\ve^2} \sum_{k=1}^n|u_k|^2w_j & =\frac{n}{\ve^2} w_j, \qin \Omega,\\
w_j&=0  \qon \pa\Omega.
\end{align*}
Hence, it follows that
\[ \int_\Om |\nabla w_j|^2+ \frac{1}{\ve^2} \int_\Om \sum_{k=1}^n|u_k|^2|w_j|^2 = \frac{n}{\ve^2}  \int_\Om |w_j|^2 \le \frac{n}{\la_1\ve^2}  \int_\Om |\nabla w_j|^2 .
\]
If $\ve>\sqrt{n/\la_1} $, then $w_j$ is a constant.
Since $w_j=0$ on $\pa \Om$, we have $w_j=0$.
\end{proof}

It is not clear   whether the conclusion of Proposition \ref{prop:soln property 1}  is   valid for arbitrary $\ve$.

\begin{remark}\label{rmk:alpha beta}
Suppose that $d_j=0$ for all $1\le j \le n$.
Since $  H^1_{g_j}(\Om; S^1)  \neq \emptyset$,  the   minimization problem
\[\alpha (g_1,\cdots, g_n) =  \inf  \Big\{ I(u_1,\cdots,u_n)  : u_j \in H^1_{g_j}(\Om; S^1) \Big\}
\]
is achieved by a unique pair $(u_1^0,\cdots,u_n^0)$  satisfying
\begin{equation}
\label{eq:u0v0 deg=0}
\left\{ \begin{aligned}
          -\Delta u_j^0 & =u_j^0  |\nabla u_j^0|^2 &&\mbox{on } \Om,\\
          |u_j^0|&=1&&\mbox{on } \Om,\\
          u_j^0&=g_j && \mbox{on } \pa \Om.
          \end{aligned}
\right.
\end{equation}
Since $d_j=0$  for  each $j=1,\cdots,n$, we have that  $u_j^0=e^{i\vp_j}$ for a harmonic function  $\vp_j$ in $\Om$.
In general, we have $\alpha  (g_1,\cdots,g_n) \geq \beta (g_1,\cdots,g_n)$ since  $(u_1^0,\cdots,u_n^0) \in \calX(g_1,\cdots,g_n) $.
An interesting question is whether $\alpha  (g_1,\cdots,g_n) = \beta (g_1,\cdots,g_n)$ or not.
The next proposition gives us a necessary  condition that they are equal.
  \end{remark}

 \begin{proposition}\label{prop:soln property 2}
 Assume that $d_j=0$ for all $j\in\{1,\cdots,n\}$.
If $\alpha  (g_1,\cdots,g_n) = \beta (g_1,\cdots,g_n)$ and it is achieved by $(u_1,\cdots,u_n)$, then $|\nabla u_j|=|\nabla u_k|$ for all $j,k\in\{1,\cdots,n\}$.
 \end{proposition}
\begin{proof}
By assumption,   $u_1,\cdots,u_n\in H_{g_1}^1(\Omega;S^1)\times\cdots\times H_{g_n}^1(\Omega;S^1)$ and
\[\beta=I(u_1,\cdots,u_n).
\]
Since $|u_j|=1$, we can write $u_j=e^{i\vp_j}$.
By plugging this in \eqref{eq:u1v1 system}, we are led to
\[-i\Delta\vp_j=\Big(\frac{1}{n}-1\Big)|\nabla\vp_j|^2+\frac{1}{n}\sum_{k\neq j}^n|\nabla\vp_k|^2 \qin \Om.
\]
Thus,  for all $j=1\cdots,n$
\[ \Delta\vp_j=0 \qand   |\nabla\vp_j|^2 = \frac{1}{n-1}\sum_{k\neq j}^n|\nabla\vp_k|^2.
\]
In particular,
\[  |\nabla\vp_j|^2 -  |\nabla\vp_1|^2  = \frac{1}{n-1} ( |\nabla\vp_1|^2  -  |\nabla\vp_j|^2 ),
\]
which implies that $ |\nabla\vp_j|^2 =  |\nabla\vp_1|^2$, or equivalently $|\nabla u_j|^2=|\nabla u_1|^2$ for all $ 1\le j \le n$.
\end{proof}

\begin{remark}
Let us  consider the case $n=2$ for simplicity with $d_1=d_2=0$.
We can write $u_j=e^{i\vp_j}$ and $g_j=e^{i\tilde{\vp}_j}$ for $j=1,2$.
According to the proof of Proposition \ref{prop:soln property 2}, if $\al(g_1,g_2)=\beta(g_1,g_2)$, then it is necessary that
\begin{equation}
\label{eq:n=2 alpha beta}
\left\{\begin{aligned}
&\Delta\vp_1=0=\Delta\vp_2 \qon \Om,\\
&|\nabla\vp_1|=|\nabla\vp_2|  \qon \Om,\\
&\vp_1=\tilde\vp_1,~\vp_2=\tilde\vp_2\qon \partial\Om.
\end{aligned}\right.
\end{equation}
So, if there are no solutions for \eqref{eq:n=2 alpha beta}, we may conclude that $\al(g_1,g_2)>\beta(g_1,g_2)$.
For instance, if $\tilvp_1$ is a constant and $\tilvp_2$ is not a constant function, then there are no solution of \eqref{eq:n=2 alpha beta}.
Another example is the case that $\tilvp_1 =\la \tilvp_2$ with $|\la| \ne 1$.
Indeed, if $\tilvp_1 =\la \tilvp_2$, then
\[ -\Delta \Big(\frac{\vp_2}{\la}\Big)=0 ~~\mbox{on}~~ \Om \qand \frac{\vp_2}{\la}=\tilvp_1 ~~\mbox{on} ~~\pa\Om.
\]
Since harmonic maps with the same boundary values are equal, we have $\vp_1=\vp_2 /\la$.
Thus, $|\nabla \vp_1|=|\nabla \vp_2|$ only when $|\la|=1$.
\end{remark}

 As a final subject of this section, we   prove Theorem \ref{thm:properties of min seq}.\\

{\bf Proof of Theorem \ref{thm:properties of min seq}:}

(i) Suppose to the contrary that there is $k\in\{1,\cdots,n\}$ such that $d_k>0$ and
\[\frac{1}{\ve^2}\int_\Om\big(1-|u_{k,\ve}|^2\big)^2dx\le C.
\]
By Theorem \ref{thm:conv},  we have $u_{k,\ve}\to u_k^*\in H^1(\Om)$ with $|u_k^*|=1$.
This implies that $u_k^* \in H^1_{g_k} (\Om;S^1) $, which is impossible since $H^1_{g_k} (\Om;S^1)=\emptyset $ if $d_k>0$.

(ii) Suppose that $d_1=\cdots=d_n=0$.
If
\[ \frac{1}{\ve^2}\int_\Om\big(1-|u_{j,\ve}|^2\big)^2dx\le C \quad \text{for all} \quad j=1,\cdots,n,
\]
 then  $u_{j,\ve}\to u_j^*$ in $H^1$ and  $|u_j^*|=1$ for each $j=1,\cdots,n$.
So, we have $\alpha(g_1,\cdots,g_n)=\beta(g_1,\cdots,g_n)$ that contradicts the assumption.

(iii)
If we set
\[X_\ve^2=\frac{1}{\ve^2}\int_\Om \big(1-|u_{1,\ve}|^2\big)^2dx \quad\mbox{and}\quad  Y_\ve^2=\frac{1}{\ve^2}\int_\Om \Big( 1- |u_{2,\ve}|^2\Big)^2dx,
\]
we are led by \eqref{eq:energy upper bd d1d2=0} and  the  Cauchy-Schwartz inequality that
\[\begin{aligned}
C&\ge \frac{1}{\ve^2}\int_\Om \big(2-|u_{1,\ve}|^2-|u_{2,\ve}|^2\big)^2  dx\\
&=X_\ve^2+Y_\ve^2+\frac{2}{\ve^2}\int_\Om \big(1-|u_{1,\ve}|^2\big) \big(1-|u_{2,\ve}|^2\big) dx\\
&\ge \big(X_\ve- Y_\ve\big)^2.
\end{aligned}
\]So, either both $X_\ve $ and $Y_\ve$ are bounded, or they are unbounded.
By (i) and (ii), at least one of $X_\ve $ and $Y_\ve$ is unbounded and we get the desired conclusion.
\qed

 \subsubsection*{Acknowledgements.}
 Jongmin Han was supported by Basic Science Research Program through
the National Research Foundation of Korea(NRF) funded by the Ministry of Education (2018R1D1A1B07042681).
Juhee Sohn was supported by the National Research Foundation of Korea(NRF) grant funded by the Korea government(MSIT) (2021R1G1A1003396).

\small
 \bibliographystyle{amsplain}

\end{document}